\pdfoutput=1
\RequirePackage{ifpdf}
\ifpdf 
\documentclass[pdftex]{sigma}
\else
\documentclass{sigma}
\fi

\usepackage{leftidx}
\usepackage{tikz-cd}
\usepackage{mathtools}
\usepackage[all]{xy}

\newtheorem{tw}{Theorem}[section]

\newtheorem{prop}[tw]{Proposition}
\newtheorem{cor}[tw]{Corollary}

{ \theoremstyle{definition}
\newtheorem{defi}[tw]{Definition}
\newtheorem{rem}[tw]{Remark}
\newtheorem{egg}[tw]{Example}}

\begin{document}
\allowdisplaybreaks

\newcommand{\arXivNumber}{1805.03066}

\renewcommand{\PaperNumber}{070}

\FirstPageHeading

\ShortArticleName{The Solution of Hilbert's Fifth Problem for Transitive Groupoids}

\ArticleName{The Solution of Hilbert's Fifth Problem \\ for Transitive Groupoids}

\Author{Pawe{\l} RA\'ZNY}

\AuthorNameForHeading{P.~Ra\'zny}

\Address{Institute of Mathematics, Faculty of Mathematics and Computer Science,\\ Jagiellonian University in Cracow, Poland}
\Email{\href{mailto:pawel.razny@student.uj.edu.pl}{pawel.razny@student.uj.edu.pl}}

\ArticleDates{Received May 11, 2018, in final form July 10, 2018; Published online July 17, 2018}

\Abstract{In the following paper we investigate the question: when is a transitive topo\-logical groupoid continuously isomorphic to a Lie groupoid? We present many results on the matter which may be considered generalizations of the Hilbert's fifth problem to this context. Most notably we present a ``solution'' to the problem for proper transitive groupoids and transitive groupoids with compact source fibers.}

\Keywords{Lie groupoids; topological groupoids}

\Classification{22A22}

\section{Introduction}
In this short paper we refine the generalization of Hilbert's fifth problem (which states that a locally Euclidean topological group is isomorphic as a topological group to a Lie group and was famously solved in~\cite{G} and~\cite{MZ}) to the case of transitive groupoids. The main results of our paper~\cite{H5oids} gave a partial solution to this problem via the following results:
\begin{tw}\label{strong}
Let $\mathcal{G}$ be a transitive topological groupoid with a smooth base $\mathcal{G}_0$, Tychonoff source fibers $\mathcal{G}_x$ and first countable space of morphisms $\mathcal{G}_1$ for which the topological groups $\mathcal{G}_x^x$ are locally Euclidean and the map $t_x$ is a~quotient map $($or equivalently has the sequence covering property$)$. Then $\mathcal{G}$ is continuously isomorphic to a unique Lie groupoid through a base preserving isomorphism.
\end{tw}

\begin{tw}\label{SE} Let $\mathcal{G}$ be a transitive topological groupoid with a smooth base $\mathcal{G}_0$ for which the spaces $\mathcal{G}_x^x$, $\mathcal{G}_1$ and $\mathcal{G}_x$ are topological manifolds and the map $t_x$ is a quotient map $($or equivalently has the sequence covering property$)$. Then $\mathcal{G}$ is continuously isomorphic to a unique Lie groupoid through a base preserving isomorphism.
\end{tw}
Using these results we were able to solve the problem for proper transitive groupoids:
\begin{tw}\label{prop} Let $\mathcal{G}$ be a proper transitive topological groupoid with a smooth base~$\mathcal{G}_0$, for which the spaces $\mathcal{G}_x^x$, $\mathcal{G}_x$ and $\mathcal{G}_1$ are topological manifolds. Then $\mathcal{G}$ is continuously isomorphic to a unique Lie groupoid through a base preserving isomorphism.
\end{tw}

In this paper we provide the solution of the general case by dropping the assumption ``$t_x$ is a~quotient map'' in Theorem~\ref{SE} (as it turns out it follows from other assumptions in Theorem~\ref{SE}) and in doing so obtain Theorem~\ref{H5S} which is the main result of this paper. We also present some versions of the Gleason--Yamabe theorem for transitive groupoids. Throughout the paper we assume manifolds to be second-countable and Haudorff (in particular they are metrizable).

\section{Preliminaries}
\subsection{Some topology}
In order to make this short paper as self-contained as possible we recall some basic topological notions and properties which are used in subsequent sections. We begin with some well known properties of quotient maps (identifications). A continuous surjective map $p\colon X\rightarrow Y$ is a~\emph{quotient map} if a subset $U$ of $Y$ is open in $Y$ if and only if~$p^{-1}(U)$ is open in~$X$. Equivalently, $Y$~is the quotient space of~$X$ with respect to the relation~$\sim$ given by $x\sim y$ if and only if there exists a point $z\in Y$ such that $x,y\in p^{-1}(z)$. Quotient maps are characterized by the following universal property:
\begin{prop}\label{PtQ} Given a quotient map $p\colon X\rightarrow Y$ and another continuous map $f\colon X\rightarrow Z$ which is constant on the fibers of~$p$ there is a unique continuous map $\overline{f}\colon Y\rightarrow Z$ such that the following diagram commutes:
\begin{equation*}
\begin{tikzcd}
X \arrow[rd, "p"] \arrow[r, "f"] & Z \\
& Y.\arrow[u, "\overline{f}"]
\end{tikzcd}
\end{equation*}
\end{prop}
We also recall the following result which will be useful in the proof of the main theorem:
\begin{tw}[{\cite[Corollary~5]{PPS}}]\label{MH} Let $f\colon M\rightarrow N$ be a continuous bijection between topological manifolds. Then $\dim(M)=\dim(N)$ and $f$ is a homeomorphism.
\end{tw}

Furthermore, we note some facts about the dimension of a separable metric space from~\cite{E}.
\begin{defi} Let $X$ be a set and $\mathcal{U}$ a family of subsets of~$X$. By the order of the family~$\mathcal{U}$ we mean the largest integer~$n$ such that the family $\mathcal{U}$ contains $n+ 1$ sets with a non-empty intersection, if no such integer exists, we say that the family $\mathcal{U}$ has order $\infty$.
\end{defi}
\begin{defi} To a separable metric space $X$ one assigns the dimension of $X$, denoted by $\dim(X)$ defined by the following conditions:
\begin{enumerate}\itemsep=0pt
\item[1)] $\dim(X) \leq n$, where $n\in\mathbb{N}$ if every finite open cover of the space X has a finite open refinement of order $\leq n$,
\item[2)] $\dim(X)=n$ if $\dim(X) \leq n$ and it is not true that $\dim(X) \leq n- 1$,
\item[3)] $\dim(X)=\infty$ if $\dim(X)\geq n$ for any $n\in\mathbb{N}$
\end{enumerate}
We also put $\dim_x(X)=\inf\{\dim(U) \,|\, U \text{ neighbourhood of } x\}$.
\end{defi}
\begin{rem} The invariant defined above is called the covering dimension of~$X$. We recall that, in the separable metric case all three definitions of dimension (small inductive, large inductive and covering; see~\cite{E}) coincide by \cite[Theorem~1.7.7]{E}. Hence, we simply call this invariant the dimension of $X$ and restate the theorems from~\cite{E} in a manner suitable to this convention.
\end{rem}
\begin{tw}[{\cite[Theorem~1.1.2]{E}}] If $X$ is a separable metric space and $A\subset X$ is a subspace then $\dim(A)\leq \dim(X)$.
\end{tw}
\begin{tw}[{\cite[Theorem~1.5.3]{E}\label{DCS}}] If a separable metric space $X$ can be represented as the union of a sequence $\{F_1, F_2,, \dots\}$ of closed subspaces such that $\dim{F_i}\leq n$ for $i\in\mathbb{N}$ then $\dim(X) \leq n$.
\end{tw}

It is also well known that the dimension of a manifold is equal to its dimension as a separable metric space and hence no ambiguity can arise (this is actually a consequence of the previous two theorems and the fact that $\dim(\mathbb{R}^n)=n$; see \cite[Theorem~1.8.2]{E} for the latter).

In the final section of this paper we will be using the following important theorem:
\begin{tw}[Gleason--Yamabe theorem for compact groups] \label{GY}Let $G$ be a compact Hausdorff topological group. For every neighbourhood $U$ of the identity of $G$ there exists a compact normal subgroup $K$ contained in $U$ such that $G\slash K$ is isomorphic as a topological group to a Lie group.
\end{tw}
The proof of this theorem can be found, e.g., in \cite[Theorem~1.4.14]{TT}.

\subsection{Groupoids}
We give a brief recollection of some basic notions concerning groupoids. Let us start by giving the definition:
\begin{defi} A groupoid $\mathcal{G}$ is a small category in which all the morphisms are isomorphisms. Let us denote by $\mathcal{G}_0$ the set of objects of this category (also called the base of $\mathcal{G}$) and by $\mathcal{G}_1$ the set of morphisms of this category. This implies the existence of the following five structure maps:
\begin{enumerate}\itemsep=0pt
\item[1)] the source map $s\colon \mathcal{G}_1\rightarrow\mathcal{G}_0$ which associates to each morphism its source;
\item[2)] the target map $t\colon \mathcal{G}_1\rightarrow\mathcal{G}_0$ which associates to each morphism its target;
\item[3)] the identity map $\operatorname{Id}\colon \mathcal{G}_0\rightarrow\mathcal{G}_1$ which associates to each object the identity over that object;
\item[4)] the inverse map $i\colon \mathcal{G}_1\rightarrow\mathcal{G}_1$ which associates to each morphism its inverse;
\item[5)] the multiplication (composition) map $\circ\colon \mathcal{G}_2\rightarrow\mathcal{G}_1$ which associates to each composable pair of morphisms its composition ($\mathcal{G}_2$ is the set of composable pairs).
\end{enumerate}
 A groupoid endowed with topologies on $\mathcal{G}_1$ and $\mathcal{G}_0$ which make all the structure maps continuous is called a topological groupoid. If additionally $\mathcal{G}_0$ and $\mathcal{G}_1$ are smooth manifolds, the source map is a surjective submersion and all the structure maps are smooth, then the topological groupoid is called a Lie groupoid.
\end{defi}
\begin{rem} Note that the identity map and the target map restricted to the image of the identity map are inverse to each other and so the identity map is an embedding. Hence, we can identify $\mathcal{G}_0$ with the image of the identity structure map.
\end{rem}
We denote by $\mathcal{G}_x$ the fibers of the source map (source fibers) and by $\mathcal{G}^x$ the fibers of the target map (target fibers). We also denote by $\mathcal{G}^y_x$ the set of morphisms with source~$x$ and target~$y$. For a Lie groupoid $\mathcal{G}_x$, $\mathcal{G}^x$ and $\mathcal{G}^y_x$ are all closed embedded submanifolds of~$\mathcal{G}_1$. What is more~$\mathcal{G}^x_x$ are Lie Groups. We also denote the image of~$\mathcal{G}_0$ through the identity structure map by $\operatorname{Id}_{\mathcal{G}}$.
\begin{defi} A morphism of groupoids is a pair $(F,f)\colon \mathcal{G}\rightarrow\mathcal{H}$ where $F\colon \mathcal{G}_1\rightarrow\mathcal{H}_1$ and $f\colon \mathcal{G}_0\rightarrow\mathcal{H}_0$ are functions which commute with the structure maps. If in addition~$\mathcal{G}$ and~$\mathcal{H}$ are topological (resp.\ Lie) groupoids and both $F$ and $f$ are continuous (resp.\ smooth) then~$(F,f)$ is called a~continuous (resp. smooth) morphism. If both~$F$ and $f$ are bijections (resp.\ homeo\-morphisms, diffeomorphisms) then $(F,f)$ is an isomorphism (resp.\ continuous isomorphism, smooth isomorphism) of groupoids. Furthermore, if f is the identity on $\mathcal{G}_0$ then the morphism~$(F,f)$ is called a base preserving morphism.
\end{defi}
Throughout this paper we are going to use several special classes of groupoids:
\begin{defi} A groupoid $\mathcal{G}$ is said to be transitive if for each pair of points $x,y\in\mathcal{G}_0$ there exists a morphism with source~$x$ and target~$y$. Conversely, if all the morphisms of~$\mathcal{G}$ are automorphisms (i.e., $s(g) = t(g)$ for all $g\in \mathcal{G}$) we say that~$\mathcal{G}$ is totally intransitive. A topological groupoid is said to be proper if the map $(s,t)\colon \mathcal{G}_1\rightarrow\mathcal{G}_0\times\mathcal{G}_0$ is proper.
\end{defi}
\begin{defi} A topological groupoid is principal if:
\begin{enumerate}\itemsep=0pt
\item[1)] the restriction of the target map to any source fiber is a quotient map (we write $t_x$ for the restriction of the target map to the source fiber over $x\in\mathcal{G}_0$);
\item[2)] for each $x\in\mathcal{G}_0$ the division map $\delta_x\colon \mathcal{G}_x\times\mathcal{G}_x\rightarrow\mathcal{G}_1$ defined by the formula $\delta_x(g,h)=g\circ h^{-1}$ is a quotient map.
\end{enumerate}
\end{defi}
\begin{rem} This notion is more commonly used in a different sense (see, e.g.,~\cite{R}). The above definition (used, e.g., in~\cite{McK1}) is convenient to the study of transitive topological groupoids whilst in the Lie case it is equivalent to local triviality of a transitive groupoid. However, the terminology of locally trivial Lie groupoids is itself rarely used due to a result of Pradines stating that every transitive Lie groupoid over a connected base is locally trivial.
\end{rem}
\begin{egg}Let G be a topological (resp. Lie) group and let M be a topological space (resp.\ smooth manifold). $M\times G\times M$ is a topological (resp.\ Lie) groupoid over M with source and target maps given by projections onto the first and third factor and composition law given by the formula
\begin{gather*}
(x,g,y)\circ (y,h,z)=(x,gh,z).
\end{gather*}
Groupoids of this form are called trivial groupoids. If $G$ is the trivial group then the above groupoid is called the pair groupoid of $M$. Such groupoids are principal topological groupoids.
\end{egg}
More examples of principal groupoids can be easily provided by the gauge groupoid construction presented in the subsequent section.
\begin{prop}[{\cite[Proposition 3.2]{McK1}}] A transitive groupoid $\mathcal{G}$ is isomorphic $($not in a continuous manner$)$ to the trivial groupoid $\mathcal{G}_0\times\mathcal{G}^x_x\times\mathcal{G}_0$.
\end{prop}

\begin{rem} It is worth noting that for a transitive topological groupoid $\mathcal{G}$ and a given source fiber $\mathcal{G}_x$ composition with $h\colon x\rightarrow y$ gives a homeomorphism $\tilde{h}\colon \mathcal{G}_y\rightarrow \mathcal{G}_x$ since it has an inverse (composition with $h^{-1}$) and conjugation by~$h$ ($hgh^{-1}$ for $g\in\mathcal{G}_x^x$) gives a continuous isomorphism of topological groups $h^*\colon \mathcal{G}^x_x\rightarrow \mathcal{G}^y_y$ since it has an inverse (conjugation by $h^{-1}$). Hence, we write~``$\mathcal{G}_x$ (resp.~$\mathcal{G}_x^x$) has property~$P$'' as shorthand for ``$\mathcal{G}_x$ (resp.\ $\mathcal{G}^x_x$) has property~$P$ for some $x\in\mathcal{G}_0$ and equivalently for any $x\in\mathcal{G}_0$'' whenever this remark can be applied. Moreover, since $t_y(\tilde{h})=t_x$ and $\delta_y(\tilde{h},\tilde{h})=\delta_x$ we write~``$t_x$~(resp.~$\delta_x$) has property $P$'' as shorthand for~``$t_x$~(resp.~$\delta_x$) has property $P$ for some $x\in\mathcal{G}_0$ and equivalently for any $x\in\mathcal{G}_0$'' whenever this remark can be applied.
\end{rem}

We are going to use the following theorem concerning groupoid morphisms with principal source:
\begin{tw}[{\cite[Proposition 1.21]{McK1}}] \label{AFPG} Let $\mathcal{G}$ be a principal topological groupoid and $\mathcal{G}'$ be any topological groupoid. Additionally, let $(\phi,f)\colon \mathcal{G}\rightarrow\mathcal{G}'$ be a morphism of groupoids $($not necessarily continuous$)$. If there is an $x\in\mathcal{G}_0$ such that $\phi_x\colon \mathcal{G}_x\rightarrow\mathcal{G}'_{f(x)}$ is continuous then $(\phi,f)$ is continuous.
\end{tw}

The final notions we need to introduce here are subgroupoids and quotient groupoids:
\begin{defi} A subgroupoid of $\mathcal{G}$ is a pair of subset $\mathcal{G}'_1\subset\mathcal{G}_1$ and $\mathcal{G}'_0\subset\mathcal{G}_0$, such that $s(\mathcal{G}'_1)\subset\mathcal{G}'_0$, $t(\mathcal{G}'_1)\subset\mathcal{G}'_0$, $Id(\mathcal{G}'_0)\subset\mathcal{G}'_1$ and $\mathcal{G}'_1$ is closed under the composition and inversion structure maps. A subgroupoid $\mathcal{G}'$ of $\mathcal{G}$ is called wide if $\mathcal{G}'_0=\mathcal{G}_0$. A subgroupoid $\mathcal{N}$ of $\mathcal{G}$ is said to be normal if it is wide and for any $g\in\mathcal{N}_1$ and any $h\in\mathcal{G}_1$ satisfying $s(h)=s(g)=t(g)$ we have $hgh^{-1}\in\mathcal{N}_1$.
\end{defi}
To introduce the notion of a quotient topological groupoid we first need do specify the underlying groupoid structure on the appropriate quotient space. Since in this paper we are only interested in quotients by totally intransitive groupoids, we provide an appropriately restricted definition (see \cite{McK1} for a more general discussion of this topic).
\begin{defi} Let $\mathcal{N}$ be a totally intransitive normal subgroupoid of the groupoid $\mathcal{G}$. Define an equivalence relation $\sim$ on $\mathcal{G}_1$ by $g\sim h$ if and only if there exists $n_1,n_2\in\mathcal{N}$ such that $n_1gn_2=h$. We then define the quotient groupoid as $\mathcal{G}\slash\mathcal{N}:=(\mathcal{G}_1\slash\sim,\mathcal{G}_0)$ with the structure maps induced from $\mathcal{G}$.
\end{defi}
\begin{egg} Simplest examples of normal subgroupoids and quotient groupoids are provided by the trivial groupoids. A normal subgroupoid $\mathcal{N}$ of a trivial groupoid $\mathcal{G}=M\times G\times M$ is of the form $\bigcup\limits_{x\in M} \{ x\}\times N\times \{ x\}$ for some normal subgroup $N$ of $G$ whilst the quotient $\mathcal{G}\slash\mathcal{N}$ is simply $M\times (G\slash N)\times M$.
\end{egg}
It was proven in \cite{BH} that if $\mathcal{G}$ is a topological groupoid then $\mathcal{G}\slash\mathcal{N}$ can be endowed with a~unique topology such that it makes $\mathcal{G}\slash\mathcal{N}$ into a topological groupoid, the projection $\pi\colon \mathcal{G}\rightarrow\mathcal{G}\slash\mathcal{N}$ is continuous and for every continuous morphism of topological groupoids $(\phi,f)\colon \mathcal{G}\rightarrow\mathcal{G'}$ such that $\phi (\mathcal{N})\subset \operatorname{Id}_{\mathcal{G}'}$ there is a unique continuous morphism $(\overline{\phi},f)\colon \mathcal{G}\slash\mathcal{N}\rightarrow\mathcal{G}'$ satisfying $\overline{\phi}\circ\pi=\phi$ (the last condition is called the universal property of topological quotient groupoids). From now on quotient groupoids of a topological groupoid will be considered with this topology. It is important to note that this topology need not coincide with the quotient topology of $\mathcal{G}_1\slash\sim$ (cf.~\cite{McK1}). This will give rise to minor difficulties when proving the Gleason--Yamabe theorem for transitive groupoids.

We conclude this section by recalling an important technical result from \cite{H5oids} which also finds a use in the current paper:
\begin{prop}\label{eq} Let $\mathcal{G}$ be a transitive topological groupoid with $\mathcal{G}_1$ first countable. Then the following conditions are equivalent:
\begin{enumerate}\itemsep=0pt
\item[$1)$] $t_x$ is a quotient map;
\item[$2)$] $t_x$ is open;
\item[$3)$] $\delta_x$ is a quotient map;
\item[$4)$] $\delta_x$ is open.
\end{enumerate}
Furthermore, if $(s,t)\colon \mathcal{G}_1\rightarrow \mathcal{G}_0\times\mathcal{G}_0$ is a quotient map, then the above properties hold.
\end{prop}

\subsection{Cartan principal bundles}
In this section we give a brief recollection of principal bundles, Cartan principal bundles, how they relate to each other as well as some results from \cite{P} which are of key importance to the present paper. A more detailed exposition of this subject can be found in~\cite{McK1} and~\cite{P}.
\begin{defi} A Cartan principal bundle is a quadruple $(P,B,G,\pi)$, where $P$ and $B$ are topological spaces, $G$ is a topological group acting freely on $P$ and $\pi\colon P\rightarrow B$ is a surjective continuous map, with the following properties:
\begin{enumerate}\itemsep=0pt
\item[1)] $\pi$ is a quotient map with fibers coinciding with the orbits of the action of $G$ on $P$,
\item[2)] the division map $\delta\colon P_{\pi}\rightarrow G$ with domain $P_{\pi}:=\{(u,v)\in P\times P\, |\, \pi(u)=\pi(v)\}$ defined by the property $\delta(ug,u)=g$ is continuous.
\end{enumerate}
\end{defi}
We are also going to need a notion of morphism between such bundles:
\begin{defi}A morphism of Cartan principal bundles is a triple
\begin{gather*} (F,f,\phi)\colon \ (P,B,G,\pi)\rightarrow (P',B',G',\pi'),\end{gather*} where $F\colon P\rightarrow P'$ and $f\colon B\rightarrow B'$ are continuous functions and $\phi\colon G\rightarrow G'$ is a continuous morphism of topological groups such that
\begin{gather*}
\pi'\circ F=f\circ\pi, \qquad F(pg)=F(p)\phi(g)
\end{gather*}
for $p\in P$ and $g\in G$. A morphism of Cartan principal bundles is said to be base preserving if~$f$ is the identity on~$B$.
\end{defi}
\begin{defi} A Cartan principal bundle $(P,B,G,\pi)$ is called proper if the action of~$G$ on~$P$ is proper.
\end{defi}

A better known and stronger notion is the following:
\begin{defi} A principal bundle is a quadruple $(P,B,G,\pi)$, where $P$ and $B$ are topological spaces, $G$ is a topological group acting freely on $P$ and $\pi\colon P\rightarrow B$ is a surjective continuous map, with the following properties:
\begin{enumerate}\itemsep=0pt
\item[1)] the fibers of $\pi$ coincide with the orbits of the action of $G$;
\item[2)] (local triviality) there is an open covering $U_i$ of $B$ and continuous maps $\sigma_i\colon U_i\rightarrow P$ such that $\pi\circ\sigma_i=\operatorname{Id}_{U_i}$.
\end{enumerate}
A principal bundle is said to be smooth if $P$ and $B$ are smooth manifolds, $G$ is a Lie group, and the action, projection and $\sigma_i$ are smooth maps.
\end{defi}

We are now going to present important constructions from \cite{McK1} and \cite{McK2} which relate the notion of Cartan principal bundles to principal groupoids. Given a principal groupoid $\mathcal{G}$ the quadruple $(\mathcal{G}_x,\mathcal{G}_0,\mathcal{G}^x_x,t_x)$ constitutes a Cartan principal bundle for any point $x\in\mathcal{G}_0$ (this is called the \emph{vertex bundle of} $\mathcal{G}$ \emph{at} $x$). It is easy to see that given a morphism of groupoids $(F,f)\colon \mathcal{G}\rightarrow\mathcal{G}'$ the restriction of the map $F$ to $\mathcal{G}_x$ gives a morphism of bundles $F|_{\mathcal{G}_x}\colon \mathcal{G}_x\rightarrow\mathcal{G}_{f(x)}$. It is also worth noting that even though this construction is dependent on the choice of $x$ all the vertex bundles are continuously isomorphic by use of translations (cf.~\cite{McK1}). In the other direction given a Cartan principal bundle $(P,B,G,\pi)$ there exists a structure of a topological groupoid over $B$ on $(P\times P)\slash G$ (this is called the \emph{gauge groupoid of} $(P,B,G,\pi)$). Furthermore, a morphism of Cartan principal bundles
\begin{gather*}(F,f,\phi)\colon \ (P,B,G,\pi)\rightarrow (P',B',G',\pi')
\end{gather*} induces a morphism of gauge groupoids $F^*$ defined by $F^*([(u,v)])=[F(u),F(v)]$. It is apparent from the form of the induced morphisms that a base preserving morphism of Cartan principal bundles induces a base preserving morphism of the corresponding gauge groupoids and that a~base preserving morphism of principal groupoids induces a base preserving morphism of vertex bundles. We give the following theorem which was proven in~\cite{McK1}:
\begin{tw}\label{CPB} The constructions above are mutually inverse $($up to a~continuous base preserving isomorphism$)$ and give a one to one correspondence between continuous isomorphism classes of Cartan principal bundles and continuous isomorphism classes of principal groupoids.
\end{tw}
We also need the following important results from \cite{P}:
\begin{tw}[{\cite[Proposition 1.2.5]{P}}]\label{Slice2} A Cartan principal bundle $(P,B,G,\pi)$ with $P$ and $B$ Tychonoff is proper.
\end{tw}
\begin{tw}[{\cite[Section 4.1]{P}}]\label{Slice} A Cartan principal bundle $(P,B,G,\pi)$ with $G$ a Lie group and $P$ Tychonoff is locally trivial.
\end{tw}
\begin{tw}[{\cite[Proposition 1.3.2]{P}}] \label{Slice4} Given a proper Cartan principal bundle $(P,B,G,\pi)$ with $P$ Tychonoff and a normal closed subgroup $N$ of $G$, $(P\slash N,B,G\slash N,\overline{\pi})$ is also a proper Cartan principal bundle with $P\slash N$ Tychonoff.
\end{tw}
\begin{tw}[{\cite[Theorem 4.3.4]{P}}]\label{Slice3} A proper Cartan principal bundle $(P,B,G,\pi)$ with $G$ a Lie group and $P$ separable and metrizable admits an invariant metric. Hence, $P\slash G$ is metrizable. Moreover, $\dim_{[x]}(P\slash G)=\dim_x(P)-\dim(G)$.
\end{tw}

\section{Hilbert's fifth problem for transitive groupoids}
\subsection{Some preliminary theorems}
This part is devoted to establishing two theorems which are crucial to the solution to Hilberts fifth problem for groupoids given in the next subsection. We feel that the results of this section should be known, however we were unable to find a suitable reference and so we provide their proofs for the readers convenience.
\begin{tw}\label{dim} Let $f\colon M_1\rightarrow M_2$ be a continuous bijection between separable metric spaces. Additionally let $M_1$ be locally compact of dimension $m$. Then $\dim(M_2)=m$.
\begin{proof} We first prove $\dim(M_2)\leq m$. Let us take around each point $x\in M_1$ its relatively compact open neighbourhood $U_x$. Since a separable metric space is Lindel\"{o}f we can choose a countable subcover $\{U_i\}_{i\in I}$ from $\{U_x\}_{x\in M_1}$. Now $f|_{\overline{U_i}}\colon \overline{U_i}\rightarrow f(\overline{U_i})$ is a homeomorphism since for $i\in I$ the set $\overline{U_i}$ is compact. Moreover, this implies that $f(\overline{U_i})$ are closed in $M_2$ of dimension at most $m$ (since homeomorphisms preserve dimension and $\dim(\overline{U_i})\leq m$ due to the fact that they are subspaces of an $m$ dimensional space). Hence the family $f(\overline{U_i})$ satisfy the conditions of Theorem \ref{DCS} which implies that $M_2$ has dimension no greater than $m$. On the other hand at least one of $\overline{U_i}$ has to have dimension $m$ since otherwise by Theorem \ref{DCS} we have $m=\dim(M_1)\leq m-1$. Hence, at least one of $f(\overline{U_i})$ is of dimension $m$ which implies that $\dim(M_2)\geq m$.
\end{proof}
\end{tw}
\begin{tw}\label{Prod1} Let $S\subset\mathbb{R}^m$ be such that for some open $V\subset\mathbb{R}^n$ the set $V\times S$ is homeomorphic to an $(n+m)$-dimensional manifold. Then $S$ is open in $\mathbb{R}^m$.\end{tw}

\begin{proof} Let us assume that $S$ is not open in $\mathbb{R}^m$. Then there exists a point $x\in S$ such that no neighbourhood of $x$ is contained in $S$. Let us take a neighbourhood $U_{(x,y)}$ of $(x,y)$ for some $y\in V$ which is homeomorphic to an open ball and let us denote by $i\colon S\times V\rightarrow \mathbb{R}^m\times\mathbb{R}^n$ the inclusion given by the inclusions of $S$ into $\mathbb{R}^m$ and $V$ into $\mathbb{R}^n$. Then by the invariance of domain theorem $i(U_{(x,y)})\subset S\times V\subset \mathbb{R}^m\times\mathbb{R}^n$ is open in $\mathbb{R}^m\times\mathbb{R}^n$. But this cannot be the case as it contains the point $(x,y)$ and it cannot contain any neighbourhood of $(x,y)$ from the product basis of $\mathbb{R}^m\times\mathbb{R}^n$ (since then $S$ would contain the projection of that set onto $\mathbb{R}^m$ which would be an open neighbourhood of $x$).
\end{proof}

\begin{rem} We would like to note that even though this result seems natural and is somewhat expected it is not trivial since there exists more than one topological space $X$ such that $X\times\mathbb{R}$ is homeomorphic to $\mathbb{R}^4$ (cf.~\cite{B}, where it is shown that $X$ doesn't even have to be a manifold).
\end{rem}
\begin{cor}\label{prod2} There is no $S\subset\mathbb{R}^k$ such that for some open $V\subset\mathbb{R}^n$ the set $V\times S$ is homeomorphic to an $(n+m)$-dimensional manifold for some $m>k$.\end{cor}

\begin{proof} We apply the previous theorem to the set $\{0\}^{m-k}\times S\subset\mathbb{R}^m$ and note that this set cannot be open in $\mathbb{R}^m$.
\end{proof}

\subsection{The solution to Hilbert's fifth problem for transitive groupoids}
Let us state our main theorem:
\begin{tw} Let $\mathcal{G}$ be a transitive topological groupoid such that $\mathcal{G}_0$ is a smooth manifold, $\mathcal{G}_x$~and~$\mathcal{G}^x_x$ are topological manifolds. Then the map~$t_x$ is a quotient map.
\end{tw}
\begin{proof} To prove this let us consider the following diagram:
\begin{equation*}
\begin{tikzcd}
\mathcal{G}_x \arrow[rd, "\pi"] \arrow[r, "t_x"] & \mathcal{G}_0 \\
& \mathcal{G}_x\slash\mathcal{G}^x_x.\arrow[u, "\overline{t_x}"]
\end{tikzcd}
\end{equation*}
Note that since the fibres of $\pi$ and $t_x$ coincide, $\overline{t_x}$ is a~bijection. We shall prove that $\mathcal{G}_x\slash\mathcal{G}_x^x$ is a~topological manifold and this by Theorem~\ref{MH} will imply that $\overline{t_x}$ is a homeomorphism which in turn implies that $t_x$ is a quotient map. First of all we observe that $\mathcal{G}_x\slash\mathcal{G}_x^x$ is Hausdorff (given two points $y,z\in\mathcal{G}_x\slash\mathcal{G}_x^x$ they can be separated by the inverse images through $\overline{t_x}$ of the open sets separating $\overline{t_x}(y)$ and $\overline{t_x}(z)$). Furthermore, $\mathcal{G}_x\slash\mathcal{G}_x^x$ is locally compact since given a point \smash{$y\in\mathcal{G}_x\slash\mathcal{G}_x^x$} the image of a compact neighbourhood of some point in $\pi^{-1}(y)$ is a compact neighbourhood of $y$ (we use here the fact that~$\pi$ as a projection onto the orbit space of a group action is open). Local compactness and being Hausdorff imply that $\mathcal{G}_x\slash\mathcal{G}_x^x$ is Tychonoff. We also note that $\mathcal{G}_x\slash\mathcal{G}_x^x$ is second countable (the countable basis is given by the images of some countable basis of $\mathcal{G}_x$ through $\pi$). The group $G^x_x$ under our assumptions is a Lie group due to the solution to the classical Hilberts fifth problem. It is also apparent that $(\mathcal{G}_x,\mathcal{G}_x\slash\mathcal{G}_x^x,\mathcal{G}_x^x,\pi)$ is a Cartan principal bundle (since $\pi$ is a quotient map and $\delta(g,h)=h^{-1}g$ must be continuous since $\mathcal{G}_1$ is a~topological groupoid). Hence, by Theorems~\ref{Slice} and~\ref{Slice2} this bundle is proper and locally trivial. Theorem~\ref{Slice3} implies that $\mathcal{G}_x\slash\mathcal{G}_x^x$ is metrizable with $\dim(\mathcal{G}_x\slash\mathcal{G}_x^x)=\dim(\mathcal{G}_x)-\dim(\mathcal{G}_x^x)$ which in turn implies that $\dim(\mathcal{G}_0)= \dim(\mathcal{G}_x)-\dim(\mathcal{G}_x^x)$ (by Theorem~\ref{dim} and the fact that a~second countable metric space is separable).

We will now show that $\mathcal{G}_x\slash\mathcal{G}_x^x$ is locally Euclidean. Let us fix a point $y\in\mathcal{G}_x\slash\mathcal{G}_x^x$ along with its open neighbourhood $S$ such that the bundle $(\mathcal{G}_x,\mathcal{G}_x\slash\mathcal{G}_x^x,\mathcal{G}_x^x,\pi)$ is trivial when restricted to $S$. Using local compactness we can assume without loss of generality that $S$ is relatively compact. This allows us to treat $S$ as a subset of $\mathcal{G}_0$ since then $\overline{t_x}|_S\colon S\rightarrow\overline{t_x}(S)$ is a homeomorphism (since~$\overline{t_x}|_{\overline{S}}$ is a continuous bijection with compact domain). If $\overline{t_x}(S)$ is a neighbourhood of~$\overline{t_x}(y)$ in~$\mathcal{G}_0$ then we are done (by taking the preimage through $\overline{t_x}$ of a neighbourhood of $\overline{t_x}(y)$ homeomorphic to an open ball contained in~$\overline{t_x}(S)$). Let us consider the set $\overline{t_x}(S)\cap U$ for some neighbourhood~$U$ of~$\overline{t_x}(y)$ homeomorphic to an open ball (this set is now homeomorphic to a~subset of $\mathbb{R}^{\dim(\mathcal{G}_0)}$). Using the commutativity of the diagram above it is apparent that
\begin{gather*}
t_x|_{t_x^{-1}(\overline{t_x}(S)\cap U)}\colon \ t_x^{-1}(\overline{t_x}(S)\cap U)\rightarrow \overline{t_x}(S)\cap U
\end{gather*}
is locally trivial (i.e., $(t_x^{-1}(\overline{t_x}(S)\cap U),\overline{t_x}(S)\cap U,\mathcal{G}^x_x,t_x|_{t_x^{-1}(\overline{t_x}(S)\cap U)})$ is a principle bundle). Hence, for any neighbourhood $V$ of the identity in $\mathcal{G}^x_x$ we have that $(\overline{t_x}(S)\cap U)\times V$ is homeomorphic to an open subset in $\mathcal{G}_x$ and hence a manifold of the same dimension as $\mathcal{G}_x$. Now by using Theorem~\ref{Prod1} we get $\overline{t_x}(S)\cap U$ is open in $U$ and conclude that $\overline{t_x}(S)$ is indeed a neighbourhood of $\overline{t_x}(y)$.
\end{proof}

From this and Theorem~\ref{strong} it is visible that the assumption that $t_x$ is a quotient map is superfluous in Theorem \ref{SE} and we get the desired result:
\begin{tw}\label{H5S} Let $\mathcal{G}$ be a transitive topological groupoid with a smooth base $\mathcal{G}_0$ for which the space $\mathcal{G}_1$ is first-countable and the spaces $\mathcal{G}_x^x$, $\mathcal{G}_x$ are topological manifolds. Then $\mathcal{G}$ is continuously isomorphic to a unique Lie groupoid through a base preserving isomorphism.
\end{tw}

\begin{rem} In particular the previous theorem holds if $\mathcal{G}_1$ is a topological manifold.
\end{rem}
\begin{rem} We wish to note that if the Hilbert--Smith conjecture (cf.~\cite{TT}) proves valid then the assumption on $\mathcal{G}_x^x$ is unnecessary. However, using this method we see no way of weakening any more assumptions in order to get a result similar to Theorem~\ref{strong}.
\end{rem}

\section{Gleason--Yamabe theorem for proper transitive groupoids}
 Our goal in this section is to find an appropriate generalization of Theorem \ref{GY} to transitive groupoids. Let us first put our result in the most general way:

\begin{tw}[Gleason--Yamabe theorem for transitive groupoids]\label{GYS} Let $\mathcal{G}$ be a transitive topological groupoid with $\mathcal{G}_1$ first-countable, $\mathcal{G}_0$ a smooth manifold, $\mathcal{G}_x$ Hausdorff and locally compact, $\mathcal{G}^x_x$ compact and $t_x$ a quotient map. For every neighbourhood~$U$ of the identity subspace $\operatorname{Id}_{\mathcal{G}}$ and every point $x\in\mathcal{G}_0$ there exists a closed totally intransitive normal subgroupoid $\mathcal{K}$ such that:
\begin{enumerate}\itemsep=0pt
\item[$1)$] $\mathcal{K}^x_x$ is compact and contained in $U\cap\mathcal{G}^x_x$,
\item[$2)$] $\mathcal{G}\slash\mathcal{K}$ is a Lie groupoid.
\end{enumerate}
\end{tw}

\begin{proof} First let us note that for a given neighbourhood $U$ of the identity space $\operatorname{Id}_{\mathcal{G}}$ and a fixed point $x\in\mathcal{G}_0$ by Theorem \ref{GY} there is a compact normal subgroup $K\subset\mathcal{G}^x_x$ such that $K$ is contained in $U$ and $\mathcal{G}^x_x\slash K$ is a Lie group. We use this group to create the desired groupoid $\mathcal{K}$. We put $\mathcal{K}^x_x=K$ and $\mathcal{K}^y_y=gKg^{-1}$ for some morphism $g\colon x\rightarrow y$. The above construction does not depend on the choice of the morphism $g$ since given another morphism $\tilde{g}\colon x\rightarrow y$ there exists an element $h\in\mathcal{G}^x_x$ such that $g=\tilde{g}h$ (namely, $h=\tilde{g}^{-1}g$). Then
\begin{gather*}
gKg^{-1}=\tilde{g}hKh^{-1}\tilde{g}^{-1}=\tilde{g}K\tilde{g}^{-1},
\end{gather*}
$\mathcal{K}$ is a topological groupoid with the induced topology. It is obvious from the construction that~$\mathcal{K}$ is also a normal subgroupoid.

Note that $\mathcal{G}_x$ is Tychonoff since it is locally compact and Hausdorff. By Proposition \ref{eq} we conclude that $\mathcal{G}_1$ is a principal groupoid and so $(\mathcal{G}_x,\mathcal{G}_0,\mathcal{G}^x_x,t_x)$ is a Cartan principal bundle. Moreover, it is a proper Cartan principal bundle due to Theorem~\ref{Slice2} and so by Theorem~\ref{Slice4} the quadruple $(\mathcal{G}_x\slash\mathcal{K}^x_x,\mathcal{G}_0,\mathcal{G}^x_x\slash\mathcal{K}^x_x,\overline{t_x})$ is also a proper Cartan principal bundle and so we can consider its gauge groupoid $\mathcal{H}$. We will show that $\mathcal{H}$ is continuously isomorphic to $\mathcal{G}\slash\mathcal{K}$.

Since $\mathcal{H}$ and $\mathcal{G}\slash\mathcal{K}$ are isomorphic (not necessarily in a continuous manner) by a base preserving isomorphism to $\mathcal{G}_0\times(\mathcal{G}^x_x\slash\mathcal{K}^x_x)\times\mathcal{G}_0$ and $\mathcal{G}$ is isomorphic (not necessarily in a continuous manner) by a base preserving isomorphism to $\mathcal{G}_0\times\mathcal{G}^x_x\times\mathcal{G}_0$ we have a morphism of groupoids $\phi\colon \mathcal{G}\rightarrow\mathcal{H}$ (the quotient morphism). We note that $\phi_x\colon \mathcal{G}_x\rightarrow\mathcal{G}_x\slash\mathcal{K}^x_x$ is then equal to the quotient map and hence continuous. This combined with Theorem~\ref{AFPG} implies that $\phi$ is in fact a continuous groupoid morphism. By the universal property of quotient groupoids we have that $\overline{\phi}\colon \mathcal{G}\slash\mathcal{K}\rightarrow\mathcal{H}$ is continuous.

On the other hand, let us observe that by the universal property of the quotient map the map $\overline{\phi_x}^{-1}\colon \mathcal{G}_x\slash\mathcal{K}_x^x\rightarrow (\mathcal{G}\slash\mathcal{K})_x$ is also continuous since it is equal to the map $\overline{\pi}_x$ induced by the groupoid projection map $\pi\colon \mathcal{G}_1\rightarrow(\mathcal{G}\slash\mathcal{K})_1$ restricted to $\mathcal{G}_x$: 
\begin{equation*}
\begin{tikzcd}
\mathcal{G}\arrow[rd, "\pi"] \arrow[r, "\phi"] & \mathcal{H} \\
& \mathcal{G}\slash\mathcal{K},\arrow[u, "\overline{\phi}"]
\end{tikzcd}
\quad
\begin{tikzcd}
\mathcal{G}_x\arrow[rd, "\pi_x"] \arrow[r, "\phi_x"] & \mathcal{G}_x\slash\mathcal{K}_x^x \\
& (\mathcal{G}\slash\mathcal{K})_x.\arrow[u, "\overline{\phi_x}"]
\end{tikzcd}
\end{equation*}
This again by Theorem \ref{AFPG} implies that $\overline{\phi}^{-1}$ is continuous and hence $\mathcal{H}$ and $\mathcal{G}\slash\mathcal{K}$ are conti\-nuous\-ly isomorphic.

We now show that $\mathcal{H}$ satisfies all the assumptions of Theorem~\ref{strong} and so it is a Lie groupoid and consequently $\mathcal{G}\slash\mathcal{K}$ is also a Lie groupoid. The fact that $\mathcal{H}$ is principal (as a gauge groupoid of a Cartan principal bundle) implies that~$t'_x$ is a quotient map (where $t'$ denotes the target map in $\mathcal{H}$). By Theorem~\ref{Slice4} we have $\mathcal{H}_x=\mathcal{G}_x\slash\mathcal{K}_x^x$ is Tychonoff. It is also locally compact as an image of a locally compact space through a continuous open map. As a quotient of a locally compact space by a group action it is locally compact and first countable (since the projection onto the orbit space is open and~$\mathcal{G}_x$ is locally compact and first countable). This also implies that $\mathcal{H}_1:=(\mathcal{H}_x\times\mathcal{H}_x)\slash\mathcal{H}^x_x$ is first countable. Finally, we note that $\mathcal{H}_0=\mathcal{G}_0$ is a smooth manifold.

We end the proof by noting that $\mathcal{K}$ is closed in $\mathcal{G}$ as it is equal to $\pi^{-1}(\operatorname{Id}_{\mathcal{H}})$.
\end{proof}

We also note that under such general assumptions we cannot demand that $\mathcal{K}$ is either compact or contained in $U$ as is shown by the following simple example:
\begin{egg}
Let us consider the trivial groupoid $\mathcal{G}:=\mathbb{R}\times\mathbb{Z}_p\times\mathbb{R}$, where $\mathbb{Z}_p$ denotes the additive group of $p$-adic integers. Let us now take the open set:
\begin{gather*}
 U:=\bigcup\limits_{n\in\mathbb{N}} (-n,n)\times B_p\left(0,\frac{1}{p^n}\right)\times (-n,n),
\end{gather*}
where $ B_p(0,r)$ denotes the $p$-adic ball of radius $r$ centered at zero. It is apparent that normal totally intransitive subgroupoids of~$\mathcal{G}$ are of the form $\bigcup\limits_{x\in\mathbb{R}} \{x\}\times N\times \{x\}$ for some subgroup~$N$ of~$\mathbb{Z}_p$ and that such a subgroupoid can be contained in~$U$ only if~$N$ is trivial. The above example also highlights the fact that the groupoid~$\mathcal{K}$ depends on the choice of~$x$.
\end{egg}

It is also noteworthy that the $\mathcal{G}_0$ has to be smooth for the theorem to work. The following counterexample shows that assuming that $\mathcal{G}_0$ is a topological manifold is insufficient:
\begin{egg} Let $\mathcal{G}_0$ be a topological manifold which does not admit any smooth structure (e.g., the celebrated $E_8$ 4-manifold). We then take $\mathcal{G}$ to be the pair groupoid over~$\mathcal{G}_0$. It is apparent that despite satisfying the assumptions of Theorem~\ref{GYS} (except for the smoothness of~$\mathcal{G}_0$) the thesis does not hold for this groupoid (since the identity subspace is the only totally intransitive groupoid).
\end{egg}

Theorem \ref{GYS} leads us to the following two corollaries:
\begin{tw}[Gleason--Yamabe theorem for proper transitive groupoids] Let $\mathcal{G}$ be a proper transitive topological groupoid with $\mathcal{G}_1$ first-countable, $\mathcal{G}_0$ a smooth manifold, $\mathcal{G}_x$ Hausdorff and locally compact. For every neighbourhood~$U$ of the identity subspace $\operatorname{Id}_{\mathcal{G}}$ and every point $x\in\mathcal{G}_0$ there exists a closed totally intransitive normal subgroupoid $\mathcal{K}$ such that:
\begin{enumerate}\itemsep=0pt
\item[$1)$] $\mathcal{K}^x_x$ is compact and contained in $U\cap\mathcal{G}^x_x$,
\item[$2)$] $\mathcal{G}\slash\mathcal{K}$ is a Lie groupoid.
\end{enumerate}
\end{tw}
\begin{proof} Since $\mathcal{G}^x_x=(s,t)^{-1}(x,x)$ it is compact. Moreover, since $(s,t)$ is proper it is also closed and hence it is a quotient map. This by Proposition \ref{eq} gives us that $t_x$ is a quotient map as well, which in turn together with Theorem \ref{GYS} proves the desired result.
\end{proof}

\begin{tw}[Gleason--Yamabe theorem for transitive groupoids with compact source fibres] Let $\mathcal{G}$ be a transitive topological groupoid with $\mathcal{G}_1$ first-countable, $\mathcal{G}_0$ a smooth manifold, $\mathcal{G}_x$ Hausdorff and compact. For every neighbourhood $U$ of the identity subspace $\operatorname{Id}_{\mathcal{G}}$ and every point $x\in\mathcal{G}_0$ there exists a closed totally intransitive normal subgroupoid $\mathcal{K}$ such that:
\begin{enumerate}\itemsep=0pt
\item[$1)$] $\mathcal{K}^x_x$ is compact and contained in $U\cap\mathcal{G}^x_x$,
\item[$2)$] $\mathcal{G}\slash\mathcal{K}$ is a Lie groupoid.
\end{enumerate}
\end{tw}
\begin{proof} $\mathcal{G}^x_x$ is compact as a closed subspace of the compact space $\mathcal{G}_x$. Moreover $t_x$ is a quotient map since it is closed as a map with compact domain. We finish the proof by applying Theorem~\ref{GYS}.
\end{proof}

\pdfbookmark[1]{References}{ref}
\LastPageEnding


\begin{thebibliography}{99}
\footnotesize\itemsep=0pt

\bibitem{B}
Bing R.H., The cartesian product of a certain non-manifold and a line is
 {$E_{4}$}, \href{https://doi.org/10.1090/S0002-9904-1958-10160-3}{\textit{Bull. Amer. Math. Soc.}} \textbf{64} (1958), 82--84.

\bibitem{BH}
Brown R., Hardy J.P.L., Topological groupoids. {I}. {U}niversal constructions,
 \href{https://doi.org/10.1002/mana.19760710123}{\textit{Math. Nachr.}} \textbf{71} (1976), 273--286.

\bibitem{E}
Engelking R., Dimension theory, \textit{North-Holland Mathematical Library},
 Vol.~19, North-Holland Publishing Co., Amsterdam~-- Oxford~-- New York,
 PWN~-- Polish Scientific Publishers, Warsaw, 1978.

\bibitem{G}
Gleason A.M., Groups without small subgroups, \href{https://doi.org/10.2307/1969795}{\textit{Ann. of Math.}}
 \textbf{56} (1952), 193--212.

\bibitem{McK1}
Mackenzie K.C.H., Lie groupoids and {L}ie algebroids in differential geometry,
 \href{https://doi.org/10.1017/CBO9780511661839}{\textit{London Mathematical Society Lecture Note Series}}, Vol.~124, Cambridge
 University Press, Cambridge, 1987.

\bibitem{McK2}
Mackenzie K.C.H., General theory of {L}ie groupoids and {L}ie algebroids,
 \href{https://doi.org/10.1017/CBO9781107325883}{\textit{London Mathematical Society Lecture Note Series}}, Vol.~213, Cambridge
 University Press, Cambridge, 2005.

\bibitem{MZ}
Montgomery D., Zippin L., Small subgroups of finite-dimensional groups,
 \href{https://doi.org/10.2307/1969796}{\textit{Ann. of Math.}} \textbf{56} (1952), 213--241.

\bibitem{P}
Palais R.S., On the existence of slices for actions of non-compact {L}ie
 groups, \href{https://doi.org/10.2307/1970335}{\textit{Ann. of Math.}} \textbf{73} (1961), 295--323.

\bibitem{PPS}
Pasike E.E., Petunin Yu.I., Savkin V.I., Continuous bijective mappings in
 topological and {B}anach manifolds, \href{https://doi.org/10.1007/BF01098343}{\textit{J.~Math. Sci.}} \textbf{58}
 (1992), 286--29.

\bibitem{H5oids}
Ra\'zny P., On the generalization of {H}ilbert's fifth problem to transitive
 groupoids, \href{https://doi.org/10.3842/SIGMA.2017.098}{\textit{SIGMA}} \textbf{13} (2017), 098, 10~pages,
 \href{https://arxiv.org/abs/1710.11440}{arXiv:1710.11440}.

\bibitem{R}
Renault J., A groupoid approach to {$C^{\ast} $}-algebras, \href{https://doi.org/10.1007/BFb0091072}{\textit{Lecture
 Notes in Math.}}, Vol.~793, Springer, Berlin, 1980.

\bibitem{TT}
Tao T., Hilbert's fifth problem and related topics, \href{https://doi.org/10.1090/gsm/153}{\textit{Graduate Studies in
 Mathematics}}, Vol.~153, Amer. Math. Soc., Providence, RI, 2014.

\end{thebibliography}
\end{document}